\documentclass[a4paper,12pt]{amsart}

\usepackage[english]{babel}
\usepackage[utf8]{inputenc}
\allowdisplaybreaks  

\usepackage{amssymb,xypic,amscd}

\usepackage{mathtools}
\usepackage{etoolbox}

\usepackage[colorlinks,linkcolor=Red3,citecolor=Red3,urlcolor=Cyan3]{hyperref}

\usepackage{enumitem}

\usepackage{parskip}

\usepackage[svgnames,x11names,dvipsnames]{xcolor}

\usepackage{bm}

\newtheorem{definition}[equation]{Definition}
\newtheorem{theorem}[equation]{Theorem}

\newtheorem{corollary}[equation]{Corollary}
\newtheorem{lemma}[equation]{Lemma}
\newtheorem{proposition}[equation]{Proposition}

\theoremstyle{remark}
\newtheorem{remark}[equation]{Remark}
\theoremstyle{remark}

\numberwithin{equation}{section} 

\newcommand{\bD}{{\mathbb{D}}}

\newcommand{\bC}{{\mathbb{C}}}

\newcommand{\bZ}{{\mathbb{Z}}}
\newcommand{\bN}{{\mathbb{N}}}

\newcommand{\cK}{{\mathcal{K}}}

\newcommand{\bQ}{{\mathbb{Q}}}

\newcommand{\res}{{\operatorname{res}}}

\title[The Lerch-type zeta function of a recurrence sequence]{The Lerch-type zeta function of a recurrence sequence of arbitrary degree}

\author[Á. Serrano~Holgado]{Álvaro Serrano Holgado}
\email{Alvaro\_ Serrano@usal.es}

\author[L.~M.~Navas~Vicente]{Luis Manuel Navas Vicente}
\email{navas@usal.es}

\thanks{Research of the second author is supported by grant PID2021-124332NB-C22 of the MICINN (Spain).}
\address{Departamento de Matem\'aticas and IUFFYM, Universidad de
Salamanca,  Plaza de la Merced 1-4
        \\
        37008 Salamanca. Spain.
        \\
         Tel: +34 923294460. 
}



\begin{document}

\maketitle

\begin{abstract}
We consider the series $\sum_{n=1}^{\infty} z^{n} (a_{n} + x)^{-s}$ where $a_{n}$ satisfies a linear recurrence of arbitrary degree with integer coefficients. Under appropriate conditions, we prove that it can be continued to a meromorphic function on the complex $s$-plane. Thus we may associate a Lerch-type zeta function $\varphi(z,s,x)$ to a general recurrence. This subsumes all previous results which dealt only with the ordinary zeta and Hurwitz cases and degrees $2$ and $3$. Our method generalizes a formula of Ramanujan for the classical Hurwitz-Riemann zeta functions. We determine the poles and residues of $\varphi$, which turn out to be polynomials in $x$. In addition we study the dependence of $\varphi$ on $x$ and $z$.
\end{abstract}

\section{Introduction}
\label{sec:intro}

The original inspiration for this paper lies in the result proved independently in \cite{Egami} and \cite{Navas}, showing that the Dirichlet series $\sum_{n=1}^{\infty} F_{n}^{-s}$, where $\{F_{n}\}$ is the Fibonacci sequence, has a meromorphic continuation to the complex plane. The pole set and residues were also determined, as well as the rationality of finite values at negative integers and formulas for values at positive integers. Similar results for quadratic recurrences were studied in~\cite{ElsnerShimomuraShiokawa,Kamano, Nakamura, Komatsu0, Silverman}. The analogous zeta function for a cubic recurrence was recently studied in~\cite{SmajlovicTribonacci}. In~\cite{Serrano} we generalized all of these results to a recurrence of arbitrary degree. Versions for multiple Lucas zeta functions have been given in~\cite{MeherRout1, MeherRout2}.

In addition,~\cite{SmajlovicLucasHurwitz} considers a Hurwitz-type zeta function, that is to say, defined by a series of the form $\sum_{n=1}^{\infty} (a_{n} + x)^{-s}$, where $\{a_{n}\}$ is a general real Lucas quadratic sequence.

In the classical theory of special functions, the Hurwitz and Lerch zeta functions provided useful extensions of the Riemann zeta function, satisfying various functional relations, differential equations, and algebraic properties (see for example~\cite{Ap-Lerch}).

Having proved the analytic continuation of a recurrence Dirichlet series in~\cite{Serrano}, we have been led naturally to consider the analogous Lerch-type series $\sum_{n=1}^{\infty} z^{n} (a_{n} + x)^{-s}$ for a linear recurrence sequence of any degree. The results presented here regarding analytic continuation contain the special cases mentioned earlier. Interestingly, in some senses this ``recurrence Lerch function'' actually behaves better than its classical analog, while in other ways the opposite is true.

We give a brief overview of the structure of the paper.

In Section~\ref{sec:first} we recall the basic facts which we need about linear recurrence sequences over the integers. We also summarize the result of~\cite{Serrano}
regarding the zeta function of a recurrence and introduce the series which is the main subject of this paper,
\[
	\varphi(z,s,x) = \sum_{n=1}^{\infty} z^{n} (a_{n} + x)^{-s}.
\]
In Section~\ref{sec:ramanujan-formula} we reduce the study of $\phi(z,s,x)$ to the case $x = 0$, providing an analog (Theorem~\ref{t:ramanujan-formula} and Corollary~\ref{c:ramanujan-hurwitz}) of a formula of Ramanujan for the Hurwitz zeta function in terms of the Riemann zeta function, namely
\[
	\zeta(s,1-x)=\sum_{j=0}^{\infty}\frac{(s)_j}{j!}\zeta(s+j)x^j.
\]
Following this strategy, in Section~\ref{sec:x=0} we deal with the case $x = 0$, proving that $\phi(z,s,0)$ has a meromorphic continuation to the whole $s$-plane (Theorem~\ref{t:continuation-x-0}) and calculating its poles and residues.

Combining the results of the previous sections, the general case of $\phi(z,s,x)$ is studied in Section~\ref{sec:continuation}. Its meromorphic continuation is shown in Theorem~\ref{t:continuation-lerch}, along with the determination of the pole set and residues (Proposition~\ref{p:residues-general}).

Finally in Section~\ref{sec:z-x} we also study $\phi(z,s,x)$ as a function of $x$ (as a series in $x$ it converges inside the unit disk) and also $z$. In Theorem~\ref{t:continuation-z} we show that it also has a meromorphic continuation  to the whole $z$-plane, and compute its poles, along with the $z$-residues (Proposition~\ref{p:residues-z}).

\section{Background and basic properties}
\label{sec:first}

In \cite{Serrano}, we studied the zeta function defined by a linear recurrence sequence $\{a_n\}$ over the integers, defined by $\sigma=\Re(s)>0$ by the Dirichlet series
\[
	\varphi(s)=\sum_{n=1}^{\infty}\frac{1}{a_n^s}.
\]

Before going on to study the Hurwitz-type and Lerch-type functions associated to a linear recurrence sequence in much the same way, let us give a brief summary of our setting and previous results, from which we will build up.

\begin{definition}
We say that a sequence $\{a_n\}$ of integers is a \textbf{linear recurrence sequence} over $\bZ$ if there is a \emph{monic} polynomial $Q(x)\in\bZ[x]$ such that, if $\nabla a_n=a_{n+1}$, $Q(\nabla)a_n=0\:\:\:\forall n\in\bN$.
\end{definition}

A LRS $\{a_n\}$ has a \textbf{minimal polynomial}, which is the monic polynomial $P(x)$ of minimal degree such that $P(\nabla)$ annihilates $\{a_n\}$.

For LRS over $\bZ$ we have the following result, which is a generalizaton of Binet's formula for the Fibonacci sequence (see \S\textbf{1.1.6} in \cite{Everest}, for exampe):

\begin{proposition}\label{p:Binet}
Let $\{a_n\}$ be a LRS over $\bZ$ with minimal polynomial $P(x)\in\bZ[x]$. Let $K$ be the splitting field of $P(x)$ over $\bQ$ and $\alpha_1,...,\alpha_r$ the roots of $P(x)$ in $K$, each with multiplicity $m_i$. There are polynomials $\lambda_i(x),...,\lambda_r(x)$, with $\lambda_i(x)\in\bQ(\alpha_i)[x]$, with $\deg\lambda_i\leq m_i-1$, such that
\[
	a_n=\lambda_1(n)\alpha_1^n+...+\lambda_r(n)\alpha_r^n\quad\forall n\in\bN.
\]
\end{proposition}

In what follow we will assume that the minimal polynomial $P(x)$ of our LRS $\{a_n\}$ is irreducible, so that in particular all its roots are simple and the polynomials $\lambda_i(x)$ of Proposition~\ref{p:Binet} are actually constants $\lambda_i\in\bQ(\alpha_i)$.

We will also assume that the minimal polynomial $P(x)$ of the LRS has a \emph{dominat root} ($\alpha_1$, without loss of generality), that is, $\alpha_1>1$ and $\alpha_1>|\alpha_i|$ for $i=2,...,r$. Under this hypothesis (and changing $\{a_n\}$ to $\{-a_n\}$ if it were necessary), $\{a_n\}$ is eventually strictly increasing and therefore eventually strictly increasing and positive, that is, there is some $n_0\in\bN$ such that $\{a_{n_0+n}\}_n$ is a strictly increasing sequence of positive integers. Since a finite number of the terms in the sum
\[
	\sum_{n=1}^{\infty}\frac{1}{a_n^s},
\]
and in the Dirichlet series we will study here, only change the total sum by an entire function, we can assume without loss of generality that $\{a_n\}$ is already strictly increasing and positive.

In this setting, if the expression of Binet's formula for $a_n$ is
\[
	a_n=\lambda_1\alpha_1^n+...+\lambda_r\alpha_r^n,
\]
we proved in \cite{Serrano} that:
\begin{enumerate}[wide, labelwidth=!, labelindent=0pt, label=$\bullet$]
\item The Dirichlet series
\begin{equation}\label{eq:def-zeta}
	\varphi(s)=\sum_{n=1}^{\infty} \frac{1}{a_n^s},
\end{equation}
whose abscissae of convergence and absolute convergence are both $\sigma=0$, has an analytic continuation to a meromorphic function of the whole $s$-plane.

\item This analytic continuation is given by the expression
\begin{equation}\label{eq:formula-zeta}
	\varphi(s)=\sum_{k_{\leq r-1}=0}^{\infty, k_{\leq r-2}}\Lambda_{k_{\leq r-1}}(s)\frac{1}{\alpha_1^{s+k_1}\alpha_2^{-k_1+k_2}...\alpha_r^{-k_{r-1}}-1},
\end{equation}
where we write, for the sake of brevity, $k_{\leq j}$ for the $j$-tuple $(k_1,...,k_j)$, $\Lambda_{k_{\leq r-1}}(s)$ for the function
\[
	\Lambda_{k_{\leq r-1}}(s)=\binom{-s,k_{\leq r-2}}{k_{\leq r-1}}\lambda_1^{-s+k_1}\lambda_2^{k_1-k_2}...\lambda_r^{k_{r-1}}
\]
and
\[
	\binom{-s,k_{\leq r-2}}{k_{\leq r-1}}=\binom{-s}{k_1}\binom{k_1}{k_2}...\binom{k_{r-2}}{k_{r-1}}.
\]

\item The poles of $\varphi(s)$ are the points
\begin{equation}\label{eq:poles-zeta}
	s_{n,k_{\leq r-1}}=\frac{\log|\alpha_1^{-k_1}\alpha_2^{k_1-k_2}...\alpha_r^{k_{r-1}}|}{\log\alpha_1}+i\frac{\arg(\alpha_2^{k_1-k_2}...\alpha_r^{k_{r-1}})+2\pi n}{\log\alpha_1}, \quad {n, k_i\in\bZ,\atop k_1\geq ...\geq k_{r-1}\geq 0.}
\end{equation}

\item The residue of $\varphi(s)$ at a pole $s=s_0$ is
\begin{equation}\label{eq:residues-zeta}
	\sum_{(n,k_{\leq r-1})\in\bm\tau(s_0)}\frac{1}{\alpha_1}\Lambda_{k_{\leq r-1}}(s_0),
\end{equation}
where $\bm\tau(s_0)$ is the set of $r$-tuples $(n,k_1,...,k_{r-1})$ such that $s_{n,k_{\leq r-1}}=s_0$ (we do not know, in general, that this set has only one $r$-tuple for each pole, although it is always finite).

\item If $m\in\bN$ and $-m$ is not a pole of $\varphi(s)$, $\varphi(-m)$ is a rational number.
\end{enumerate}

We will now study the Hurwitz-type version of the zeta function of a recurrence sequence, defined as a Dirichlet series, for $0\leq x<1$, by
\begin{equation}\label{eq:def-hurwitz}
	\varphi(s,x)=\sum_{n=1}^{\infty}\frac{1}{(a_n+x)^s},
\end{equation}
and the Lerch-type function, defined by
\begin{equation}\label{eq:def-lerch}
	\varphi(z,s,x)=\sum_{n=1}^{\infty}\frac{z^n}{(a_n+x)^s}.
\end{equation}

These are the natural adaptations to our setting of the classical \emph{Hurwitz zeta function}, \emph{Lerch zeta function} and \emph{Lerch transcendent}. Clearly,
\[
	\varphi(s,x)=\varphi(1,s,x),
\]
and, since the techniques that we use can be applied equally to the series \eqref{eq:def-hurwitz} and \eqref{eq:def-lerch}, we will deal mainly with the Lerch-type zeta function, and the results will apply to the Hurwitz-type function by setting $z=1$.

We begin by studying the abscissae of convergence and absolute convergence of these Dirichlet series. The formulae we use can be found in \cite{Landau}. We will study $\varphi(z,s,x)$ as first only as a function of $s$, assuming $z$ and $x$ fixed. We will later study it as a function of the other variables. We set, for ease of notation in some formulae,
\[
	z=r e^{2\pi i \xi},
\]
that is, $r=|z|$ and $\xi=\frac{1}{2\pi}\arg(z)$.

\begin{proposition}\label{p:convergence-lerch}
The Dirichlet series \eqref{eq:def-lerch} that defines the Lerch-type zeta function $\varphi(z,s,x)$ has $\sigma=\frac{\log r}{\log\alpha_1}$ as abscissa of both convergence and absolute convergence.
\end{proposition}

\begin{proof}
First of all, note that, since $\log (a_n+x)\approx n\log\alpha_1$ as $n\rightarrow\infty$,
\[
	\limsup_{n\rightarrow\infty}\frac{\log n}{\log(a_n+x)}=\lim_n\frac{\log n}{n\log \alpha_1}=0,
\]
and since this number bounds the distance between both abscissae, both are the same. Therefore, we only have to compute the abscissa of absolute convergence, which is
\[
	\begin{cases}
	\displaystyle\limsup_n\frac{\log (r+...+r^n)}{\log(a_n+x)} \quad &\text{if }\sum r^n=+\infty \\
	 \\
	\displaystyle\limsup_n\frac{\log(r^n+r^{n+1}+...)}{\log(a_n+x)} \quad &\text{if }\sum r^n<+\infty\end{cases}
\]

Since $\sum r^n$ converges if and only if $r<1$, those are the two cases we need to study.

If $r>1$, $r+...+r^n=\frac{r-r^{n+1}}{1-r}$, and
\[
	\limsup_n\frac{\log(r+...+r^n)}{\log(a_n+x)}=\lim_n\frac{n\log r}{n\log\alpha_1}=\frac{\log r}{\log\alpha_1}.
\]
If $r<1$, $r^n+r^{n+1}+...=\frac{r^n}{1-r}$, and we get the same result.

Finally, if $r=1$, $r+...+r^n=n$ and the upper limit is $0$, which is precisely $\frac{\log r}{\log\alpha_1}$.
\end{proof}

Here we see the main difference between our Lerch-type Dirichlet series and the Dirichlet series that defines the classical Lerch transcendent,
\[
	\sum_{n=1}^{\infty}\frac{z^n}{(n+x)^s},
\]
which is that the latter does not converge \emph{anywhere} if $r>1$. This difference allows us to study our functions $\varphi(z,s,x)$ and $\varphi(s,x)$ using the same techniques.

\begin{corollary}\label{c:convergence-hurwitz}
The Hurwitz-type Dirichlet series $\varphi(s,x)$ defined in \eqref{eq:def-hurwitz} has $\sigma=0$ as abscissa of both convergence and absolute convergence.
\end{corollary}

\section{A formula of Ramanujan}
\label{sec:ramanujan-formula}

As a first step towards proving the analytic continuation of $\varphi(z,s,x)$, we will prove a formula that is, in this setting, equivalent to a formula of Ramanujan relating the Hurwitz and Riemann zeta functions, namely
\[
	\zeta(s,1-x)=\sum_{j=0}^{\infty}\frac{(s)_j}{j!}\zeta(s+j)x^j,
\]
which appears as equation (15) in \cite{Ramanujan}, and where $(s)_j$ is the Pochhammer symbol, or raising factorial,
\[
	(s)_j=s(s+1)...(s+j-1).
\]

Our version of this formula will allow us to simplify the study of $\varphi(z,s,x)$ to that of $\varphi(z,s,0)$. Before that, however, let us state and prove a bound for the binomial coefficient that will be useful throughout the paper.

\begin{lemma}\label{l:bound-binomial}
Let $s\in\bC$ and $j\in\bN_0$. If $|s|\leq M$, then
\begin{equation}\label{eq:bound-binomial}
	\left|\binom{-s}{j}\right|\leq (-1)^j\binom{-M}{j}.
\end{equation}
In particular,
\[
	\left|\binom{-s}{j}\right|\leq (-1)^j\binom{-|s|}{j}.
\]
\end{lemma}

\begin{proof}
If $j=0$, both sides are $1$. If $j>0$,
\[
	\binom{-s}{j}=\frac{(-s)(-s-1)...(-s-j+1)}{j!}=(-1)^j\frac{s(s+1)(s+2)...(s+j-1)}{j!}.
\]

Since $|s+l|\leq |s|+l\leq M+l$ for every $l=0,...,j-1$,
\[
	\left|\binom{-s}{j}\right|\leq\frac{M(M+1)..(M+j-1)}{j!}=(-1)^j\binom{-M}{j}.
\]
\end{proof}

\begin{theorem}\label{t:ramanujan-formula}
If $\sigma>\frac{\log r}{\log\alpha_1}$, it holds
\begin{equation}\label{eq:ramanujan-formula}
	\varphi(z,s,x)=\sum_{j=0}^{\infty}\binom{-s}{j}\varphi(z,s+j,0)x^j.
\end{equation}
\end{theorem}

\begin{proof}
Since $0\leq x<1$ and $a_n\in\bN$ for every $n$, $\left|\frac{x}{a_n}\right|<1$, and we can use the binomial series to get
\[
	(a_n+x)^{-s}=a_n^{-s}\left(1+\frac{x}{a_n}\right)^{-s}=\sum_{j=0}^{\infty}\binom{-s}{j}a_n^{-s-j}x^j.
\]

Therefore,
\begin{equation}\label{eq:ramanujan-formula-series-before}
	\varphi(z,s,x)=\sum_{n=1}^{\infty}\frac{z^n}{(a_n+x)^s}=\sum_{n=1}^{\infty}\sum_{j=0}^{\infty}\binom{-s}{j}x^j\frac{z^n}{a_n^{s+j}}.
\end{equation}

The formula \eqref{eq:ramanujan-formula} will follow if we can change the order of summation in \eqref{eq:ramanujan-formula-series-before}, which will follow from the fact that it is an absolutely convergent double series. In order to prove that note that $\{|z^n a_n^{-s-j}|\}_j$ is decreasing, so it is bounded by $|z^na_n^{-s}|$.

Using this and the bound \eqref{eq:bound-binomial} with $M=|s|$ we get
\[
\begin{aligned}
	&\sum_{n=1}^{\infty}\sum_{j=0}^{\infty}\left|\binom{-s}{j}x^j\frac{z^n}{a_n^{s+j}}\right| \leq  \\
	\leq &\sum_{n=1}^{\infty}\frac{r^n}{a_n^{-\sigma}}\sum_{j=0}^{\infty}(-1)^j\binom{-|s|}{j}|x|^j= \\
	=&(1-|x|)^{-|s|}\sum_{n=1}^{\infty}\frac{r^n}{a_n^{\sigma}}.
\end{aligned}
\]

For $\sigma>\frac{\log r}{\log\alpha_1}$, this last series is absolutely convergent, so the double series \eqref{eq:ramanujan-formula-series-before} is absolutely convergent too and we can change the order of the sums:
\begin{equation}\label{eq:ramanujan-formula-series-after}
\begin{aligned}
	\varphi(z,s,x) &=\sum_{n=1}^{\infty}\sum_{j=0}^{\infty}\binom{-s}{j}x^j\frac{z^n}{a_n^{s+j}}= \\
	&=\sum_{j=0}^{\infty}\binom{-s}{j}\sum_{n=1}^{\infty}\frac{z^n}{a_n^{s+j}}
\end{aligned}
\end{equation}

If $\sigma>\frac{\log r}{\log\alpha_1}$, the series in $n$ in \eqref{eq:ramanujan-formula-series-after} converges to $\varphi(z,s+j,0)$ for every $j\geq 0$, and we get the formula \eqref{eq:ramanujan-formula}.
\end{proof}

\begin{corollary}\label{c:ramanujan-hurwitz}
	If $\sigma>0$, the Hurwitz-type Dirichlet series $\varphi(s,x)$ satisfies the formula
\begin{equation}\label{eq:ramanujan-hurwitz}
	\varphi(s,x)=\sum_{j=0}^{\infty}\binom{-s}{j}\varphi(s+j)x^j.
\end{equation}
\end{corollary}

\section{The case $x=0$}
\label{sec:x=0}

In light of \eqref{eq:ramanujan-formula}, we can base our study of the function $\varphi(z,s,x)$ on the study of $\varphi(z,s,0)$.

\begin{theorem}\label{t:continuation-x-0}
	The Lerch-type Dirichlet series $\varphi(z,s,0)$ which, as a function of $s$, is holomorphic for $\sigma>\frac{\log r}{\log\alpha_1}$, has an analytic continuation to a meromorphic function of the whole $s$-plane whose poles are all simple, and they are the points
	\begin{equation}\label{eq:poles-x-0}
		s_{n,k_1,...,k_{r-1}}^{(z)}=\frac{\log r}{\log\alpha_1}+i\frac{2\pi\xi}{\log\alpha_1}+s_{n,k_1,...,k_{r-1}}, \quad {n,k_i\in\bZ, \atop  k_1\geq ...\geq k_{r-1}\geq 0,}
	\end{equation}
where $s_{n,k_{\leq r-1}}$ are the poles poles of the zeta function $\varphi(s)$, defined in \eqref{eq:poles-zeta}. That is, the array of poles of $\varphi(z,s,0)$ is a translation of that of $\varphi(s)$.
\end{theorem}

\begin{proof}
As we did in \cite{Serrano} for the zeta function $\varphi(s)$, we begin by using the binomial series to get
\[
	a_n^{-s}=\sum_{k_{\leq r-1}=0}^{\infty, k_{\leq r-2}}\binom{-s,k_{\leq r-2}}{k_{\leq r-1}}(\lambda_1\alpha_1^n)^{-s-k_1}(\lambda_2\alpha_2^n)^{k_1-k_2}...(\lambda_r\alpha_r^n)^{k_{r-1}}.
\]
This is possible because
\[
	\left|\frac{\lambda_2\alpha_2^n+...+\lambda_r\alpha_r^n}{\lambda_1\alpha_1^n}\right|<1
\]
holds for large enough $n$, and we can assume without loss of generality that it holds for every $n$ by removing a finite number of terms of the sum if necessary, which has no effect in what we want to prove.

With this, we can write the Dirichlet series $\sum z^na_n^{-s}$ as
\begin{equation}\label{eq:x-0-series-before}
	\sum_{n=1}^{\infty}\sum_{k_{\leq r-1}}^{\infty, k_{\leq r-2}}\Lambda_{k_{\leq r-1}}(s)(z\alpha_1^{-s-k_1}\alpha_2^{k_1-k_2}...\alpha_r^{k_{r-1}})^n.
\end{equation}

Let us now show that \eqref{eq:x-0-series-before} is absolutely convergent for $\sigma>\frac{\log r}{\log\alpha_1}$, which will allow us to change the order of summation.

First of all, by using once again the bound \eqref{eq:bound-binomial} with $M=|s|$, we get that
\[
	\left|\Lambda_{k_{\leq r-1}}(s)\right|\leq (-1)^{k_1}\binom{-|s|,k_{\leq r-2}}{k_{\leq r-1}}\lambda_1^{-\sigma-k_1}|\lambda_2|^{k_1-k_2}...|\lambda_r|^{k_{r-1}}.
\]

Secondly, we can assume that the sequence
\[
	\frac{|\lambda_2\alpha_2^^n|+...+|\lambda_r\alpha_r^n|}{\lambda_1\alpha_1^n}
\]
is strictly decreasing and less than $1$ (this happens except for a finite number of terms, and they can be removed without changing what we want to prove).

Finally, the sequence
\[
	\left(1-\frac{|\lambda_2\alpha_2^n|+...+|\lambda_r\alpha_r^n|}{\lambda_1\alpha_1^n}\right)^{-|s|}
\]
has limit $1$, so it is in particular bounded: there is some $K_s>0$ such that
\[
	\left(1-\frac{|\lambda_2\alpha_2^n|+...+|\lambda_r\alpha_r^n|}{\lambda_1\alpha_1^n}\right)^{-|s|}<K_s\quad\forall n\in\bN.
\]

Using these three facts, and the binomial series once again, we get that
\[
\begin{aligned}
	&\sum_{n=1}^{\infty}\sum_{k_{\leq r-1}=0}^{\infty, k_{\leq r-2}}\left|\Lambda_{k_{\leq r-1}}(s)(z\alpha_1^{-s-k_1}\alpha_2^{k_2-k_1}...\alpha_r^{k_{r-1}})^n\right|\leq \\
	\leq &\sum_{n=1}^{\infty}\sum_{k_{\leq r-1}=0}^{\infty, k_{\leq r-2}}(-1)^{k_1}\binom{-|s|,k_{\leq r-2}}{k_{\leq r-1}}r^n(\lambda_1\alpha_1^n)^{-\sigma-k_1}|\lambda_2\alpha_2^n|^{k_1-k_2}...|\lambda_r\alpha_r^n|^{k_{r-1}} =\\
	=&\sum_{n=1}^{\infty} \lambda_1^{-\sigma}\alpha_1^{-n\sigma}r^n\left(1-\frac{|\lambda_2\alpha_2^n|+...+|\lambda_r\alpha_r^n|}{\lambda_1\alpha_1^n}\right)^{-|s|}\leq \lambda_1^{-\sigma}K_s\sum_{n=1}^{\infty}(r\alpha_1^{-\sigma})^n.
\end{aligned}
\]

If $\sigma>\frac{\log r}{\log\alpha_1}$, $r\alpha_1^{-\sigma}<1$ and this last sum in $n$ converges, which means that \eqref{eq:x-0-series-before} is absolutely convergent. By changing the order of summation and making the sum in $n$, which is possible if $\sigma>\frac{\log r}{\log\alpha_1}$, we get that $\varphi(z,s,0)$ can be expressed as
\begin{equation}\label{eq:x-0-series-after}
	\varphi(z,s,0)=\sum_{k_{\leq r-1}=0}^{\infty, k_{\leq r-2}}\Lambda_{k_{\leq r-1}}(s)\frac{1}{z^{-1}\alpha_1^{s+k_1}\alpha_2^{-k_1+k_2}...\alpha_r^{-k_{r-1}}-1}.
\end{equation}

So far we have only proved this expression for $\sigma>\frac{\log r}{\log\alpha_1}$. Let us now see that it defines a meromorphic function of $s$ in the whole $s$-plane by showing that the series \eqref{eq:x-0-series-after} converges normally on compacts sets not containing any of the points \eqref{eq:poles-x-0}, which will consequently be its poles.

Let, then, $s$ vary in one such compact set. In particular, $s$ is bounded: there is some $M>0$ such that $|s|\leq M$. Now, for $k_1\gg 0$ it holds
\[
\begin{aligned}
	\left|z^{-1}\alpha_1^{s+k_1}\alpha_2^{-k_1+k_2}...\alpha_r^{-k_{r-1}}\right| &\geq r^{-1}\alpha_1^{\sigma+k_1}|\alpha_2|^{-k_1+k_2}...|\alpha_r|^{k_{r-1}}-1 > \\
	&> r^{-1}\alpha_1^{\sigma+k_1-1}|\alpha_2|^{-k_1+k_2}...|\alpha_r|^{k_{r-1}}
\end{aligned}
\]

In addition, note that $|\sigma|\leq |s|\leq M$ and therefore $\alpha_1^{-\sigma}\leq\alpha_1^M$ and
\[
	\lambda_1^{-\sigma}\leq\begin{cases}
	\lambda_1^M \quad &\text{if }\lambda_1\geq 1 \\
	\lambda_1^{-M} \quad &\text{if }\lambda_1\leq 1\end{cases}
\]
In any case, both $\alpha_1^{-\sigma}$ and $\lambda_1^{-\sigma}$ are uniformly bounded in our compact set (let us use $K_{\lambda_1}$ for the bound of $\lambda_1^{-\sigma}$). Using both of these facts and the bound \eqref{eq:bound-binomial} for the binomial coefficients, we get that there is some $k_0\gg 0$ such that, if $k_1\geq k_0$ and $s$ varies in this compact set,
\[
\begin{aligned}
	&\left|\left|\Lambda_{k_{\leq r-1}}(s)\frac{1}{z^{-1}\alpha_1^{s+k_1}...\alpha_r^{-k_{r-1}}}\right|\right| \leq \\
	& \leq K_{\lambda_1}\alpha_1^{M+1}(-1)^{k_1}\binom{-M,k_{\leq r-2}}{k_{\leq r-1}}(\lambda_1\alpha_1)^{-k_1}|\lambda_2\alpha_2|^{k_1-k_2}...|\lambda_r\alpha_r|^{k_{r-1}}.
\end{aligned}
\]

Now, observe that
\[
\begin{aligned}
	&\sum_{k_1=k_0}^{\infty}\sum_{k_2=0}^{k_1}...\sum_{k_{r-1}=0}^{k_{r-2}}K_{\lambda_1}\alpha_1^{M+1}(-1)^{k_1}\binom{-M,k_{\leq r-2}}{k_{\leq r-1}}(\lambda_1\alpha_1)^{-k_1}...|\lambda_r\alpha_r|^{k_{r-1}} \leq \\
	&\leq K_{\lambda_1}\alpha_1^{M+1}\sum_{k_1=0}^{\infty}...\sum_{k_{r-1}=0}^{k_{r-2}}(-1)^{k_1}\binom{-M,k_{\leq r-2}}{k_{\leq r-1}}(\lambda_1\alpha_1)^{-k_1}...|\lambda_r\alpha_r|^{k_{r-1}} = \\
	&= K_{\lambda_1}\alpha_1^{M+1}\left(1-\frac{|\lambda_2\alpha_2|+...+|\lambda_r\alpha_r|}{\lambda_1\alpha_1}\right)^{-M}<+\infty,
\end{aligned}
\]
so the series \eqref{eq:x-0-series-after} is in fact normally convergent on every compacts set not containing any of the points \eqref{eq:poles-x-0}, which means it is a meromorphic function in the whole $s$-plane with those points as its poles, which are simple.
\end{proof}

\begin{remark}\label{re:residues-x-0}
As we said, the poles of $\varphi(z,s,0)$ are a translation of the poles of the zeta funciton $\varphi(s)$. Furthemore, this translation does not affect the computation of the residues at those poles: it is very easy to check that, if $s_0$ is a pole of $\varphi(s)$ and $s_0^{(z)}$ is the corresponding pole of $\varphi(z,s,0)$, then
\[
	\res_{s=s_0^{(z)}}\varphi(z,s,0)=\res_{s=s_0}\varphi(s).
\]
Since we already know the residues of $\varphi(s)$, which are given by \eqref{eq:residues-zeta}, we have automatically the residues of $\varphi(z,s,0)$.
\end{remark}

\section{Analytic continuation}
\label{sec:continuation}

We are now finally in a position to prove that the Lerch-type Dirichlet series $\varphi(z,s,x)$ has a continuation to the whole plane as a function of $s$.

\begin{theorem}\label{t:continuation-lerch}
The Lerch-type Dirichlet series $\varphi(z,s,x)$, which, as a function of $s$, is holomorphic in the half-plane $\sigma>\frac{\log r}{\log\alpha_1}$, has an analytic continuation to a meromorphic function of the whole $s$-plane whose poles are all simple and lie at the points
\begin{equation}\label{eq:poles-general}
	-j+s_{n,k_1,...,k_{r-1}}^{(z)}, \qquad j,n,k_i\in\bZ, \:\: j\geq 0, \:\: k_1\geq...\geq k_{r-1}\geq 0.
\end{equation}
\end{theorem}

\begin{proof}
By Theorem \ref{t:ramanujan-formula} we have that, for $\sigma>\frac{\log r}{\log\alpha_1}$, $\varphi(z,s,x)$ verifies formula \eqref{eq:ramanujan-formula}, that is,
\[
	\varphi(z,s,x)=\sum_{j=0}^{\infty}\binom{-s}{j}\varphi(z,s+j,0)x^j.
\]

By Theorem \ref{t:continuation-x-0}, each $\varphi(z,s+j,0)$ is a meromorphic function of the whole $s$-plane and has poles $-j+s_{n,k_{\leq r-1}}^{(z)}$, with fixed $j$. We will now prove that formula \eqref{eq:ramanujan-formula} actually defines a meoromorphic function of the whole plane by once again showing that the series
\[
	\sum_{j=0}^{\infty}\binom{-s}{j}\varphi(z,s+j,0)x^j
\]
converges normally on compact sets not containing any of the points \eqref{eq:poles-general}.

Let $s$ vary in one of these compact sets $\cK$. In particular, $|s|\leq M$ for some $M>0$, and also $|\sigma|\leq |s|\leq M$. Since $\varphi(z,s+j,0)$ is an absolutely convergent Dirichlet series for $\sigma>-j+\frac{\log r}{\log\alpha_1}$, there is some $j_0\geq 0$ such that, for every $j\geq j_0$, $\varphi(z,s+j,0)$ is defined by an absolutely convergent Dirichlet series at every point in $\cK$ (for example, taking $j_0\geq M+\frac{\log r}{\log\alpha_1}$ suffices).

For $j\geq j_0$, we have that
\[
	|\varphi(z,s+j,0)|\leq\sum_{n=1}^{\infty}\left|\frac{z^n}{(a_n+x)^{s+j}}\right|=\sum_{n=1}^{\infty}\frac{r^n}{(a_n+x)^{\sigma+j}}=\varphi(r,\sigma+j,0).
\]

Convergent Dirichlet series with positive coefficients are \emph{decreasing} on the real axis, which means in particular that
\[
	\varphi(r,\sigma+j,0)\leq \varphi(r,\sigma+j_0,0) \quad\forall j\geq j_0.
\]

Finally, $\varphi(r,\sigma+j_0,0)$ is uniformly bounded when $s$ varies in our compact set by some constant $M'$: we can take $M'=\varphi(r,\sigma_0+j_0,0)$, where $\sigma_0=\min\{\sigma\:\vert\: s\in\cK\}$. In any case, we have for $j\geq j_0$ the bound
\[
	\left|\left|\binom{-s}{j}\varphi(z,s+j,0)x^j\right|\right|\leq M'(-1)^j\binom{-M}{j}|x|^j.
\]
Now, and keeping in mind that $0\leq x<1$,
\[
\begin{aligned}
	&\sum_{j=j_0}^{\infty} \left|\left|\binom{-s}{j}\varphi(z,s+j,0)x^j\right|\right| \leq M'\sum_{j=j_0}^{\infty}(-1)^j\binom{-M}{j}|x|^j \leq \\
	&\leq  M'\sum_{j=0}^{\infty}(-1)^j\binom{-M}{j}|x|^j=M'(1-|x|)^{-M},
\end{aligned}
\]
which gives us the normal convergence of \eqref{eq:ramanujan-formula} $\cK$. Since $\cK$ is any compact set not containing any of the points \eqref{eq:poles-general}, we get the desired result.
\end{proof}

\begin{corollary}\label{c:continuation-hurwitz}
The Hurwtiz-type Dirichlet series $\varphi(s,x)$, which, as a function of $s$, is holomorphic in the half plane $\sigma>0$, has an analytic continuation to a meromorphic function of the whole $s$-plane whose poles are all simple and lie at the points
\begin{equation}\label{eq:poles-hurwtiz}
	-j+s_{n,k_1,...,k_{r-1}}, \qquad j,n,k_i\in\bZ, \:\: j\geq 0, \:\: k_1\geq...\geq k_{r-1}\geq 0.
\end{equation}
\end{corollary}

Furthermore, we can fairily easily say what are the residues of these zeta functions at each of the poles, by once again simplifying the problem to $x=0$, whose residues are already know (see Remark \ref{re:residues-x-0}).

\begin{proposition}\label{p:residues-general}
If $s_0$ is a pole of the Lerch-type zeta function $\varphi(z,s,x)$, the residue of $\varphi(z,s,x)$ at $s_0$ is a polynomial in the variable $x$ and, in fact, if
\[
	j_0=\left[ -\sigma_0+\frac{\log r}{\log\alpha_1}\right],
\]
we have that
\begin{equation}\label{eq:residues-general}
	\res_{s=s_0}\varphi(z,s,x)=\sum_{j=0}^{j_0}\binom{-s_0}{j}\left(\res_{s=s_0}\varphi(z,s,0)\right)x^j.
\end{equation}
\end{proposition}

\begin{proof}
The absicssa of absolute convergence of $\varphi(z,s+j,0)$ is $-j+\frac{\log r}{\log\alpha_1}$, so $\varphi(z,s+j,0)$ is holomorphic at $s_0$ whenever $\sigma_0+j>\frac{\log r}{\log\alpha_1}$. This means that only a finite number of the $\varphi(z,s+j,0)$ can have a pole at $s_0$, and in fact $\varphi(z,s+j,0)$ can have a pole at $s_0$ only if
\[
	j\leq \left[-\sigma_0+\frac{\log r}{\log\alpha_1}\right].
\]
The result and formula \eqref{eq:residues-general} follow from this.
\end{proof}

Once again, by setting $z=1$ we get the same result about the Hurwitz-type zeta function $\varphi(s,x)$.

\section{Dependence on $z$ and $x$}
\label{sec:z-x}

So far, we have studied $\varphi(z,s,x)$ as a function of the variable $s$, for suitable fixed $z$ and $x$. However, we can ask ourselves how does $\varphi(z,s,x)$ depend on the other two variables.

The variable $x$ does not affect the properties we have studied much: the abscissae of convergence of $\varphi(z,s,x)$ does not depend on $x$, and the proof of Theorem \ref{t:continuation-lerch} can be understood as to say that the formula \eqref{eq:ramanujan-formula}
\[
	\varphi(z,s,x)=\sum_{j=0}^{\infty}\binom{-s}{j}\varphi(z,s+j,0)x^j
\]
provides an analytic continuation of $\varphi(z,s,x)$ on the variable $x$ to the open unit disk $\bD$ (as a series in $x$, it is normally convergent in every compact subset of the open unit disk).

We can say a bit more about the variable $z$. First of all, the definition of $\varphi(z,s,x)$, as a series in \eqref{eq:def-lerch},
\[
	\varphi(z,s,x)=\sum_{n=1}^{\infty}\frac{z^n}{(a_n+x)^s},
\]
that we have until now understood as a Dirichlet series in the variable $s$, can also be thought of as a power series in the variable $z$. The coefficients of this power series are, for fixed $s\in\bC$ and $0\leq x<1$, the numbers
\[
	c_n=(a_n+x)^{-s},
\]
and we can easily compute its radius of convergence as
\[
	\rho^{-1}=\lim_{n\rightarrow\infty}\sqrt[n]{(a_n+x)^{-s}}=\alpha_1^{-\sigma},
\]
that is, $\rho=\alpha_1^{\sigma}$. Note that this is consistent Proposition \ref{p:convergence-lerch}, since
\[
	r<\alpha_1^{\sigma}\Leftrightarrow \sigma>\frac{\log r}{\log\alpha_1}.
\]

It turns out that $\varphi(z,s,x)$ also has an analytic continuation to a meromorphic function of the whole $z$-plane, as a function of $z$. Much of our previous work is still useful to prove this, but the details are different enough to warrant a more careful explanation.

\begin{theorem}\label{t:continuation-z}
The function $\varphi(z,s,x)$, understood as a function in the variable $z$ for $s$ and $x$ fixed, which defines a holomorphic function in the open disk $|z|<\rho$, for $\rho=\alpha_1^{\sigma}$, has an analytic continuation to a meromorphic function of the whole $z$-plane. Its poles are all simple and lie at the points
\begin{equation}\label{eq:poles-z}
	z_{j,k_{\leq r-1}}^{(s)}=\alpha_1^{s+j+k_1}\alpha_2^{-k_1+k_2}...\alpha_r^{-k_{r-1}}, \quad j,k_i\in\bN_0, \:\: k_1\geq ...\geq k_{r-1}.
\end{equation}
\end{theorem}

\begin{proof}
Theorem \ref{t:ramanujan-formula} holds still even if we now see $\varphi(z,s,x)$ as a function of $z$, provided we change the hypothesis from $\sigma>\frac{\log r}{\log\alpha_1}$ to $r<\alpha_1^{\sigma}$, which are equivalent. We have therefore still formula \eqref{eq:ramanujan-formula}, that is,
\[
	\varphi(z,s,x)=\sum_{j=0}^{\infty}\binom{-s}{j}\varphi(z,s+j,0)x^j.
\]

Just as before, we base our study of $\varphi(z,s,x)$ on that of $\varphi(z,s,0)$.

If we now look carefully at the proof of Theorem \ref{t:continuation-x-0}, the condition $r<\alpha_1^s$ allows us to arrive once again at equation \eqref{eq:x-0-series-after}, that is,
\[
	\varphi(z,s,0)=\sum_{k_{\leq r-1}=0}^{\infty, k_{\leq r-1}}\Lambda_{k_{\leq r-1}}\frac{1}{z^{-1}\alpha_1^{s+k_1}\alpha_2^{-k_2+k_1}...\alpha_r^{-k_r}-1}.
\]

From this, we proceeded by proving that, as a function of $s$, this series is normally convergent on compact sets not containing any of the zeroes of the denominators. As a function of $z$, these denominators have zeroes $z_{0,k_{\leq r-1}}^{(s)}$, following the notation in \eqref{eq:poles-z}. When $z$ lies in a compact sets $\cK$ not containing any of these points, in particular there is some $M>0$ such that $r=|z|\leq M$, and, as in the proof of Theorem \ref{t:continuation-x-0}, for $k_1\gg 0$ it holds
\[
\begin{aligned}
	\left|\alpha_1^{s+k_1}\alpha_2^{-k_1+k_2}...\alpha_r^{-k_{r-1}}-z\right| &\geq \alpha_1^{\sigma+k_1}|\alpha_2|^{-k_1+k_2}...|\alpha_r|^{-k_{r-1}}-r \geq \\
	&\geq \alpha_1^{\sigma+k_1}|\alpha_2|^{-k_1+k_2}...|\alpha_r|^{-k_{r-1}}-M \geq \\
	&\geq \alpha_1^{\sigma+k_1-1}|\alpha_2|^{-k_1+k_2}...|\alpha_r|^{-k_{r-1}}.
\end{aligned}
\]

Therefore, in $\cK$,
\[
\begin{aligned}
	&\left|\left|\Lambda_{k_{\leq r-1}}(s)\frac{z}{\alpha_1^{s+k_1}...\alpha_r^{-k_{r-1}}-z}\right|\right| \leq M \left|\Lambda_{k_{\leq r-1}}(s)\right| \alpha_1^{-\sigma-k_1}...|\alpha_r|^{k_{r-1}} \leq \\
	\leq & M\alpha_1^{-1} (-1)^{k_1}\binom{-|s|,k_{\leq r-2}}{k_{\leq r-1}}(\lambda_1\alpha_1)^{-\sigma-k_1}|\lambda_2\alpha_2|^{k_1-k_2}...|\lambda_r\alpha_r|^{k_{r-1}}.
\end{aligned}
\]

The series on the $k_i$ with this general term converges to
\[
	M\lambda_1^{-\sigma}\alpha_1^{-\sigma-1}\left(1-\frac{|\lambda_2\alpha_2|+...+|\lambda_r\alpha_r|}{\lambda_1\alpha_1}\right)^{-|s|},
\]
so the series \eqref{eq:x-0-series-after} also converges normally in compact sets not containing any of the points $z_{0,k_{\leq r-1}}^{(s)}$ as a function of $z$, so $\varphi(z,s,0)$ is a meromorphic function of $z$ in the whole plane.

Finally, we use formula \eqref{eq:ramanujan-formula}, that is,
\[
	\varphi(z,s,x)=\sum_{j=0}^{\infty}\binom{-s}{j}\varphi(z,s+j,0)x^j,
\]
to prove analytic continuation of the function $\varphi(z,s,x)$ in the variable $z$, similarly as in Theorem \ref{t:continuation-lerch}, in the following way: if $z$ varies in a compact set $\cK'$ not containing any of the points $z_{j,k_{\leq r-1}}^{(s)}$ (so, in particular, $r=|z|\leq M'$ for some $M'>0$), and since the radius of convergence of $\varphi(z,s+j,0)$, as a power series in $z$, is $\alpha_1^{\sigma+j}$, which goes to $\infty$ as $j$ grows, there is some $j_0\in\bN$ such that $\varphi(z,s+j,0)$ is defined by the power series \eqref{eq:def-lerch} for every $z\in\cK'$ (any $j_0$ such that $M<\alpha_1^{\sigma+j_0}$ works for this purpose). If $j\geq j_0$, we have
\[
\begin{aligned}
	|\varphi(z,s+j,0)|\leq \varphi(r,\sigma+j,0)\leq \varphi(M,\sigma+j,0)\leq \varphi(M,\sigma+j_0,0).
\end{aligned}
\]

Therefore,
\[
\begin{aligned}
	\sum_{j=j_0}^{\infty}\left|\left| \binom{-s}{j}\varphi(z,s+j,0)x^j\right|\right|&\leq \varphi(M,\sigma+j_0,0)\sum_{j=j_0}^{\infty}(-1)^J\binom{-|s|}{j}|x|^j \leq \\
	&\leq \varphi(M,\sigma+j_0,0)\sum_{j=0}^{\infty}(-1)^j\binom{-s}{j}|x|^j = \\
	&= \varphi(M,\sigma+j,0)(1-|x|)^{-|s|},
\end{aligned}
\]
so \eqref{eq:ramanujan-formula} also defines a meromorphic function of $z$ on the whole plane, with poles at the points \eqref{eq:poles-z}, which concludes the proof.
\end{proof}

As for the residues at each of this poles, we can say the following:

\begin{proposition}\label{p:residues-z}
Let $z_0\in\bC$ be a pole of $\varphi(z,s,x)$. The residue of $\varphi(z,s,x)$ at $z=z_0$ is a polynomial in $x$. More precisely, there is some $j_0\in\bN$ such that
\[
	\res_{z=z_0}\varphi(z,s,x)=\sum_{j=0}^{j_0}\binom{-s}{j}\res_{z=z_0}\varphi(z,s+j,0)x^j.
\]

Furthermore, for the case $x=0$, if, given $s\in\bC$, $\bm\tau(z_0)$ is the set of $(r-1)$-tuples $k_{\leq r-1}$ such that $z_0=\alpha_1^{s+k_1}...\alpha_r^{-k_{r-1}}$ (which is finite), then
\[
	\res_{z=z_0}\varphi(z,s,0)=-z_0\sum_{\bm\tau(z_0)}\Lambda_{k_{\leq r-1}}(s).
\]
\end{proposition}

\begin{proof}

The same part can be proved using that the exact same argument as in the proof of Proposition \ref{p:residues-general}, but changing the abscissa of absolute convergence of the Dirichlet series in $s$ for the radius of convergence of the power series in $z$.

As for the second part, note that if $z_0=z_{0,k_{\leq r-1}}^{(s)}$,
\[
	\lim_{z\rightarrow z_0}(z-z_0)\Lambda_{k_{\leq r-1}}(s)\frac{1}{z^{-1}\alpha_1^{s+k_1}...\alpha_r^{k_{r-1}}-1}=-\Lambda_{k_{\leq r-1}}(s)z_0.
\]

The result immediately follows.
\end{proof}

\end{document}